\documentclass[11pt]{article}
\usepackage{hyperref}
\hypersetup{%
  colorlinks=true,
  linkcolor=black,
  citecolor=black,
  urlcolor=blue,
  pdftitle={The large diffusion limit for the heat equation in the exterior of the unit ball with a dynamical boundary condition},
  pdfauthor={Fila, Ishige, Kawakami, Lankeit},
  pdfkeywords={heat equation; dynamical boundary condition; large diffusion limit},
  bookmarksopen=false,
} 

\usepackage[truedimen,margin=30truemm, bottom=30truemm]{geometry}

\usepackage[T1]{fontenc}
\usepackage[utf8]{inputenc}
\usepackage{latexsym,amssymb}
\usepackage{color}
\usepackage{amsmath,amsthm}

\newtheorem{Theorem}{\bf Theorem}[section]
\newtheorem{Lemma}{\bf Lemma}[section]
\newtheorem{Proposition}{\bf Proposition}[section]
\newtheorem{Corollary}{\bf Corollary}[section]
\newtheorem{Remark}{\bf Remark}[section]
\newtheorem{Example}{\bf Example}[section]
\newtheorem{Definition}{\bf Definition}[section]

\newcommand{\Bbar}{\overline{B}}
\newcommand{\set}[1]{\{#1\}}
\newcommand{\f}[2]{\frac{#1}{#2}}
\DeclareMathOperator{\argmax}{argmax}

\newenvironment{theorem}{\begin{Theorem}$\!\!\!$}{\end{Theorem}}
\newenvironment{lemma}{\begin{Lemma}$\!\!\!$}{\end{Lemma}}

\newenvironment{corollary}{\begin{Corollary}$\!\!\!$}{\end{Corollary}}

\newenvironment{definition}{\begin{Definition}$\!\!\!$}{\end{Definition}}

\numberwithin{equation}{section}

\title{The large diffusion limit for the heat equation\\
in the exterior of the unit ball\\
with a dynamical boundary condition
}
\usepackage{authblk}

\author[1]{Marek Fila\footnote{e-mail: fila@fmph.uniba.sk}}
\author[2]{Kazuhiro Ishige\footnote{e-mail: ishige@ms.u-tokyo.ac.jp}}
\author[3]{Tatsuki Kawakami\footnote{e-mail: kawakami@math.ryukoku.ac.jp}}
\author[1,4]{Johannes Lankeit\footnote{e-mail: jlankeit@math.upb.de}}

\affil[1]{Department of Applied Mathematics and Statistics, Comenius University, 
Mlynsk\'a dolina, 84248 Bratislava, Slovakia}
\affil[2]{Graduate School of Mathematical Sciences, The University of Tokyo, 
3-8-1 Komaba, Meguro-ku, Tokyo 153-8914, Japan}
\affil[3]{Applied Mathematics and Informatics Course, Faculty of Advanced Science and Technology,
Ryukoku University, 1-5 Yokotani, Seta Oe-cho, Otsu, Shiga 520-2194, Japan}
\affil[4]{Institut f\"ur Mathematik, Universit\"at Paderborn, Warburger Str.~100, 33098 Paderborn, Germany}

\date{}

\begin{document}
\maketitle

\begin{abstract}
\noindent 
We study the heat equation in the exterior of the unit ball with a linear dynamical boundary condition. 
Our main aim is to find upper and lower bounds for the rate of convergence to solutions of the Laplace equation 
with the same dynamical boundary condition as the diffusion coefficient tends to infinity.
\end{abstract}
\section{Introduction}
We consider the problem
\begin{equation}
\label{eq:1.1}
\left\{
\begin{array}{ll}
\displaystyle{\varepsilon\partial_tu_\varepsilon-\Delta u_\varepsilon=0}, &
x\in\Omega:=\{x\in{\mathbb R}^N:|x|>1\},\,\,\,t>0,\vspace{5pt}\\
\displaystyle{\partial_tu_\varepsilon+\partial_\nu u_\varepsilon=0}, & x\in\partial\Omega,\,\,\, t>0,\vspace{5pt}\\
\displaystyle{u_\varepsilon(x,0)=\varphi(x)},\qquad & x\in\Omega,\vspace{5pt}\\
\displaystyle{u_\varepsilon(x,0)=\varphi_b(x)},\qquad & x\in\partial\Omega,
\end{array}
\right.
\end{equation}
where $N\ge3$, $\Delta$ is the $N$-dimensional Laplacian (in $x$), 
$\nu$ is the exterior normal vector to $\partial \Omega$, 
$\partial_t:=\partial/\partial t$, $\partial_\nu:=\partial/\partial\nu$, 
and $(\varphi,\varphi_b)$ is a pair of measurable functions in $\Omega$ and $\partial\Omega$,
respectively.
Our aim is to study the convergence as $\varepsilon\to0$ of the solution $u_\varepsilon$ to
the solution~$u$ of the problem
\begin{equation}
\label{eq:1.2}
\left\{
\begin{array}{ll}
\displaystyle{\Delta u=0}, & x\in\Omega,\,\,\,t>0,\vspace{5pt}\\
\displaystyle{\partial_tu+\partial_\nu u=0}, & x\in\partial\Omega,\,\,\, t>0,\vspace{5pt}\\
\displaystyle{u(x,0)=\varphi_b(x)},\qquad & x\in\partial\Omega.
\end{array}
\right.
\end{equation}
For bounded domains this convergence was established in \cite{Gal} and for
the half-space $\Omega={\mathbb R}^N_+:={\mathbb R}^{N-1}\times{\mathbb R}_+$, $N\ge2$, in \cite{FIK02}. 
More recently, the following four theorems on the rate of this convergence have been proven in \cite{FIKL}.
\begin{theorem} 
\label{Theorem:1.1}
Let $\Omega={\mathbb R}^N_+$, $N\ge2$. Let $\varphi\in L^\infty(\Omega)$, $\varphi_b\in L^\infty(\partial\Omega)$,
$\mathfrak K\subset\overline{\Omega}$ compact and $0<\tau_1<\tau_2<\infty$. 
Then there exists $C>0$ such that 
$$
\sup_{\tau_1<t<\tau_2}\|u_\varepsilon(t)-u(t)\|_{L^\infty(\mathfrak K)}
\le C\varepsilon^{\frac{1}{2}},\qquad\varepsilon\in(0,1).
$$
\end{theorem}
\begin{theorem}
\label{Theorem:1.2}
 Let $\Omega={\mathbb R}^3\setminus \overline{B_1(0)}$. 
 Let $\varphi\in L^\infty(\Omega)$ be radially symmetric such that 
$\sup_{\rho\ge 1} |\rho\varphi(\rho)|<\infty$. 
Assume further that $\varphi_b$ is constant,
$\varepsilon_0\in(0,\pi^{-1/2})$ and $0<\tau_1<\tau_2<\infty$. 
Then there exists $C>0$ such that 
$$
\sup_{\tau_1<t<\tau_2} \|u_\varepsilon(t)-u(t)\|_{L^\infty(\Omega)} 
\le C\varepsilon^{\frac{1}{2}},\qquad\varepsilon\in(0,\varepsilon_0).
$$
\end{theorem}
The upper bounds from Theorems~\ref{Theorem:1.1} and \ref{Theorem:1.2} are sharp.

\begin{theorem}
\label{Theorem:1.3}
Let $\Omega={\mathbb R}^N_+$, $N\ge 2$ and $\varphi_b\equiv 0$. 
Then there exist $\varphi\in L^\infty(\Omega)$ and a compact set 
$\mathfrak K\subset {\mathbb R}^N_+\times(0,\infty)$
such that
$$ 
u_\varepsilon(x,t)-u(x,t) =  u_\varepsilon(x,t) \ge c\varepsilon^{\frac{1}{2}},
\qquad \varepsilon\in(0,\varepsilon_0),\quad(x,t)\in \mathfrak K,
$$
for some $\varepsilon_0>0$ and $c>0$.
\end{theorem}
\begin{theorem}
\label{Theorem:1.4}
Let $\Omega={\mathbb R}^3\setminus \overline{B_1(0)}$ and $\varphi_b\equiv 0$.
Then there exist a radially symmetric $\varphi\in L^\infty(\Omega)$
satisfying $\sup_{\rho\ge 1} |\rho\varphi(\rho)|<\infty$ 
and a compact set 
$\mathfrak K\subset ({\mathbb R}^3\setminus B_1(0) )\times(0,\infty)$ 
such that 
$$
 u_\varepsilon(x,t) - u(x,t) = u_\varepsilon(x,t) \ge c\varepsilon^{\frac{1}{2}},
\qquad\varepsilon\in(0,\varepsilon_0),\quad(x,t)\in \mathfrak K,
$$
for some $\varepsilon_0>0$ and $c>0$.
\end{theorem}

We see that for the half-space the rate does not depend on the dimension,
and we obtain the same rate $\varepsilon^{1/2}$ also for the exterior of a ball in ${\mathbb R}^3$,
which is a very different domain. 
The main motivation of this paper is the natural question whether or not other rates may occur. 
We show that for ${\mathbb R}^N\setminus \overline{B_1(0)}$ the rate depends on $N$.
\vspace{5pt}

Before we formulate our main results,
we introduce some notation. 
Let $\Gamma_D=\Gamma_D(x,y,t)$ be the Dirichlet heat kernel on $\Omega$.
Define
$$
[S_1(t)\phi](x):=\int_{\Omega}\Gamma_D(x,y,t)\phi(y)\, dy,
\qquad x\in\overline{\Omega},\quad t>0,
$$
for any measurable function $\phi$ in $\Omega$.
Let $P=P(x,y)$ be the Poisson kernel on $B=B(0,1):=\{x\in {\mathbb R}^N:|x|<1\}$, that is 
$$
P(x,y):=c_N\frac{1-|x|^2}{|x-y|^N},\qquad x\in \Bbar,\quad y\in\partial B\setminus\set{x},
$$
where $c_N$ is a constant to be chosen such that 
$\|P(x,\cdot)\|_{L^1(\partial B)}=1$ for $x\in B$ 
(see (2.28) in \cite{GT}). 
Then $P=P(x,y)$ satisfies as a function of $x$  
\begin{equation}\label{eq:1.3}
-\Delta_x P=0\quad\mbox{in}\quad B,
\qquad 
P(x,y)=\delta_y\quad\mbox{on}\quad\partial B,
\end{equation}
where $\delta_y$ is the Dirac measure on $\partial B=\partial\Omega$ at $y$. 
We denote by $K=K(x,y)$ the Kelvin transform of $P$ as a function of $x$ with respect to $B$, that is 
\begin{equation}
\label{eq:1.4}
K(x,y):=|x|^{-(N-2)}P\left(\frac{x}{|x|^2},y\right),
\qquad x\in\overline{\Omega},
\quad y\in\partial\Omega\setminus\set{x}.
\end{equation}
Set 
\begin{equation}
\label{eq:1.5}
\mathcal{K}(x,y,t):=K(e^tx,y),
\qquad x\in\overline{\Omega},
\quad y\in\partial\Omega,\quad t\ge 0,\quad e^tx\neq y. 
\end{equation} 
Then it follows from \eqref{eq:1.3} and \eqref{eq:1.4} that 
${\mathcal K}={\mathcal K}(x,y,t)$ as a function of $x$ and $t$ satisfies 
$$
\left\{
\begin{array}{ll}
\displaystyle{-\Delta_x {\mathcal K}}=0 & \mbox{in}\quad\Omega\times(0,\infty),\vspace{5pt}\\
\displaystyle{\partial_t{\mathcal K}+\partial_\nu {\mathcal K}}=0 & \mbox{on}\quad\partial\Omega\times(0,\infty),\vspace{5pt}\\
\displaystyle{{\mathcal K}(\cdot,y,0)=\delta_y} & \mbox{on}\quad\partial\Omega.
\end{array}
\right.
$$
For any nonnegative measurable function $\psi$ on $\partial\Omega$ and $t>0$,
we define 
\begin{equation}
\label{eq:1.6}
[S_2(t)\psi](x):=\int_{\partial\Omega}{\mathcal K}(x,y,t)\psi(y)\,d\sigma_y
\equiv \int_{\partial\Omega}K(e^tx,y)\psi(y)\,d\sigma_y,
\qquad x\in\overline{\Omega}.
 \end{equation}

We formulate the definition of a solution of \eqref{eq:1.1} 
by the use of the two integral kernels $\Gamma_D$ and $\mathcal K$.  
For simplicity, let $\varphi_b=\varphi_b(x)$ and $g=g(x,t)$ be continuous functions 
in $\partial\Omega$ and $\partial\Omega\times(0,\infty)$, respectively. 
Then the function
\begin{equation}  
\label{eq:1.7}
 w(x,t)=w(x',x_N,t)
 :=[S_2(t)\varphi_b](x)+\int_0^t[S_2(t-s)g(s)](x)\,ds
\end{equation}
can be defined for $x\in\Omega$ and $t>0$,
and it is a classical solution of the Cauchy problem for the Laplace equation with a nonhomogeneous dynamical boundary condition
\begin{equation}
\label{eq:1.8}
\left\{
\begin{array}{ll}
\displaystyle{-\Delta w=0},
& x\in\Omega,\quad t>0,\vspace{5pt}\\
\displaystyle{\partial_tw+\partial_\nu w=g},
& x\in\partial\Omega,\quad t>0,\vspace{5pt}\\
\displaystyle{w(x,0)=\varphi_b(x)},\quad
&x\in\partial\Omega.
\end{array}
\right.
\end{equation}
It follows from \eqref{eq:1.6} and \eqref{eq:1.7} that
\begin{equation}
\label{eq:1.9}
\begin{split}
\partial_tw(x,t):= & \,\,\int_{\partial\Omega}\partial_t\mathcal K(x,y,t)\varphi_b(y)\,d\sigma_y
+\int_{\partial\Omega}K(x,y)g(y,t)\,d\sigma_y\\
 & \,\,\,\,\,\,
+\int_0^t\int_{\partial\Omega}\partial_t\mathcal K(x,y,t-s)g(y,s)\,d\sigma_y\,ds, \qquad x\in\Omega,\quad t\in(0,T).
\end{split}
\end{equation}
Set 
\begin{equation}
\label{eq:1.10}
\Phi(x):=\varphi(x)-[S_2(0)\varphi_b](x),\qquad x\in\Omega. 
\end{equation}
Then the function  
$$
v_\varepsilon(x,t)
:=[S_1(\varepsilon^{-1}t)\Phi](x)
-\int_0^t[S_1(\varepsilon^{-1}(t-s))\partial_t w(s)](x)\, ds,\qquad x\in\Omega,\quad t\ge 0,
$$
satisfies 
\begin{equation}
\label{eq:1.11}
\left\{
\begin{array}{ll}
\displaystyle{\varepsilon \partial_tv_\varepsilon=\Delta v_\varepsilon -\varepsilon\partial_t w
},\quad
& x\in\Omega,\quad t>0,\vspace{5pt}\\
\displaystyle{v_\varepsilon=0},
& x\in\partial\Omega,\quad t>0,\vspace{5pt}\\
\displaystyle{v_\varepsilon(x,0)=\Phi(x)},
&x\in\Omega.
\end{array}
\right.
\end{equation}
If $g_\varepsilon(x,t):=-\partial_{\nu} v_\varepsilon(x,t)$ for $x\in\partial\Omega$, $t>0$, and
$w_\varepsilon$ is defined as in \eqref{eq:1.7} with $g_\varepsilon$ instead of $g$,
then it follows from \eqref{eq:1.8}, \eqref{eq:1.9} and \eqref{eq:1.11} that
\begin{equation}
\label{eq:1.12}
\left\{
\begin{array}{ll}
\displaystyle{\varepsilon \partial_t v_\varepsilon=\Delta v_\varepsilon-\varepsilon F_1[\varphi_b]
+\varepsilon F_2[v_\varepsilon]},
\qquad & x\in\Omega,\,\,\,t>0,\vspace{5pt}\\
\displaystyle{\Delta w_\varepsilon=0}, & x\in\Omega,\,\,\,t>0,\vspace{5pt}\\
\displaystyle{v_\varepsilon=0},\quad
\displaystyle{\partial_tw_\varepsilon+\partial_{\nu} w_\varepsilon=-\partial_{\nu}v_\varepsilon}, & x\in\partial\Omega,\,\,\, t>0,\vspace{5pt}\\
\displaystyle{v_\varepsilon(x,0)=\Phi(x)},\qquad & x\in\Omega,\vspace{5pt}\\
\displaystyle{w_\varepsilon(x,0)=\varphi_b(x)}, & x\in\partial\Omega,
\end{array}
\right.
\end{equation}
where
\begin{align}
\label{eq:1.13}
F_1[\varphi_b](x,t) &
:=\int_{\partial\Omega}\partial_t\mathcal K(x,y,t)\varphi_b(y)\,d\sigma_y,\\
\begin{split}
\label{eq:1.14}
F_2[v](x,t) &
:=\int_{\partial\Omega}K(x,y)\partial_{\nu}v(y,t)\,d\sigma_y
+\int_0^t\int_{\partial\Omega}\partial_t\mathcal K(x,y,t-s)\partial_{\nu}v(y,s)\,d\sigma_y\,ds. 
\end{split}
\end{align}
Furthermore,
 the function $u_\varepsilon:=v_\varepsilon+w_\varepsilon$ is a classical solution of \eqref{eq:1.1}. 
Motivated by this observation, 
we formulate the definition of a solution of \eqref{eq:1.1} via problem~\eqref{eq:1.12}. 
\begin{definition}
\label{Definition:1.1}
Let $\varphi$ and $\varphi_b$ be measurable functions 
in $\Omega$ and $\partial\Omega$, respectively.
Let $0<T\le\infty$ and 
$$
v_\varepsilon,\,\, w_\varepsilon\in C(\overline{\Omega}\times(0,T)),\quad \nabla v_\varepsilon\in C(\overline{\Omega}\times(0,T)). 
$$
We call $(v_\varepsilon,w_\varepsilon)$ a solution of \eqref{eq:1.12} in $\Omega\times(0,T)$
if $v_\varepsilon$ and $w_\varepsilon$ satisfy
\begin{align*}
\begin{split}
v_\varepsilon(x,t)
 & =[S_1(\varepsilon^{-1}t)\Phi](x)
-\int_0^t[S_1(\varepsilon^{-1}(t-s))F_1[\varphi_b](s)](x)\,ds\\
 & \qquad\qquad\qquad\qquad\quad\,\,\,
+\int_0^t[S_1(\varepsilon^{-1}(t-s))F_2[v_\varepsilon](s)](x)\,ds,
 \end{split}
 \notag
\\
w_\varepsilon(x,t)
&
=[S_2(t)\varphi_b](x)-\int_0^t[S_2(t-s)\partial_{\nu}v_\varepsilon(s)](x)\, ds,
\end{align*}
for $x\in\overline{\Omega}$ and $t\in(0,T)$.
Furthermore, 
for the solution $(v_\varepsilon,w_\varepsilon)$ of \eqref{eq:1.12} in $\Omega\times(0,T)$,
we call $u_\varepsilon:=v_\varepsilon+w_\varepsilon$ a solution of \eqref{eq:1.1} in $\Omega\times(0,T)$.
In the case when $T=\infty$, we call $(v_\varepsilon,w_\varepsilon)$ a global-in-time solution of \eqref{eq:1.12}
and $u_\varepsilon$ a global-in-time solution of \eqref{eq:1.1}.
\end{definition} 

We are ready to state the main results of this paper.
For $1\le r\le\infty$ and $\theta\in(0,1)$, we write 
$|\cdot|_{L^r}:=\|\cdot\|_{L^r(\partial\Omega)}$, $\|\cdot\|_{L^r}:=\|\cdot\|_{L^r(\Omega)}$ and $|\cdot|_{C^{1,\theta}}:=\|\cdot\|_{C^{1,\theta}(\partial\Omega)}$ for simplicity.
\begin{theorem}
\label{Theorem:1.5}
Let $N\ge3$, $\varphi\in L^\infty(\Omega)$ 
and $\varphi_b\in C^{1,\theta}(\partial\Omega)$ with $\theta\in(0,1)$. 
Assume
\begin{equation}
\label{eq:1.15}
M:=\sup_{x\in\Omega}|x|^{N-2}|\varphi(x)|<\infty.
\end{equation}
Then for every $\varepsilon\in(0,1)$ the problem \eqref{eq:1.12} possesses a unique global-in-time solution 
$(v_\varepsilon,w_\varepsilon)$. These solutions have 
the following properties:
\begin{itemize}
  \item[{\rm (i)}] 
 For any $T>0$ there exists $C_T>0$ such that for every $\varepsilon\in(0,1)$ and every $\varphi, \varphi_b$ as above
  \begin{equation}
  \label{eq:1.16}
\sup_{0<t<T}\,\left[\|v_\varepsilon(t)\|_{L^\infty}
+(\varepsilon^{-1}t)^{\frac{1}{2}}\|\nabla v_\varepsilon(t)\|_{L^\infty}
   +\|w_\varepsilon(t)\|_{L^\infty}\right]
   \le C_T(|\varphi_b|_{C^{1,\theta}}+M).
  \end{equation}
  Furthermore, 
  $$
 \nabla^j v_\varepsilon\in C^\infty(\Omega\times I)\cap BC(\overline\Omega\times I),
 \qquad
 \partial_t^\ell\nabla^j w_\varepsilon\in C^\infty(\Omega\times I)\cap BC(\overline\Omega\times I)
  $$
  for any bounded interval $I\subset(0,\infty)$ and $0\le \ell+j\le 1$.
  \item[{\rm (ii)}] 
  Let $T>0$, $\tau\in(0,T)$ and
  \begin{equation}
  \label{eq:1.17}
  \alpha=1\quad\mbox{for}\quad N=3,
  \qquad
  1<\alpha<2\quad\mbox{for}\quad N\ge4.
  \end{equation}
  Then there exists $C>0$ such that for every $\varepsilon\in(0,1)$ 
  \begin{align}
  \label{eq:1.18}
   & \sup_{\tau<t<T}\,\|v_\varepsilon(t)\|_{L^\infty}\le C\varepsilon^{\frac{\alpha}{2}},\\
   \label{eq:1.19}
   & \sup_{0<t<T}\,\|w_\varepsilon(t)-S_2(t)\varphi_b\|_{L^\infty}\le C\varepsilon^{\frac{\alpha}{2}}.
  \end{align}
\end{itemize}
\end{theorem}

The reason why it is natural to assume \eqref{eq:1.15} is explained in \cite[Section~7]{FIKL}.
As a corollary of Theorem~\ref{Theorem:1.5}, 
we see that the solution $u_\varepsilon=v_\varepsilon+w_\varepsilon$ of \eqref{eq:1.1} 
converges to the solution $S_2(t)\varphi_b$ of \eqref{eq:1.2}.
\begin{corollary}
\label{Corollary:1.1}
Assume the same conditions as in Theorem~{\rm\ref{Theorem:1.5}}. 
Let $\alpha$ be as in \eqref{eq:1.17} 
and
$(v_\varepsilon,w_\varepsilon)$ the solution given in Theorem~{\rm\ref{Theorem:1.5}}. 
Then $u_\varepsilon=v_\varepsilon+w_\varepsilon$ is a classical global-in-time solution of \eqref{eq:1.1}. 
Furthermore, for any $T>0$ and $\tau\in(0,T)$ there exists $C>0$ such that
$$
\sup_{\tau<t<T}\,\|u_\varepsilon(t)-S_2(t)\varphi_b\|_{L^\infty}\le C\varepsilon^{\frac{\alpha}{2}}
$$
for every $\varepsilon \in(0,1)$.
\end{corollary}

This means that for $N\ge 4$ the convergence is faster than for $N=3$.
Moreover, we obtain an estimate from below. 
In the case of $N=3$, the rates in Corollary~\ref{Corollary:1.1} and in the following
Theorem~\ref{Theorem:1.6} coincide, and if $N=4$ they can become arbitrarily close. 
\begin{theorem}
\label{Theorem:1.6}
Let $N\ge3$.
There exists a nonnegative function $\varphi\in L^\infty(\Omega)$ with \eqref{eq:1.15}
such that the following holds: 
Let $\mathfrak K$ be a compact set in $\Omega\times(0,\infty)$.
Then there exists $C>0$ such that 
for any $\varepsilon\in(0,1)$
the corresponding solution $u_\varepsilon$ of \eqref{eq:1.1} with $\varphi_b\equiv0$ satisfies 
\begin{equation}
\label{eq:1.20}
 u_\varepsilon(x,t)-[S_2(t)\varphi_b](x)=u_\varepsilon(x,t) \ge C\varepsilon^{\frac N2 -1},
\qquad\quad(x,t)\in \mathfrak K.
\end{equation}
\end{theorem}

The rest of this paper is organized as follows. 
In Section~2 we recall some properties of the Dirichlet heat kernel $\Gamma_D$ and the kernel $\mathcal K$.
Furthermore, we prepare some useful lemmata.
In Section~3, modifying the argument as in \cite{FIK02}, 
we give a proof of Theorem~\ref{Theorem:1.5}. 
In Section~4 we prove Theorem~\ref{Theorem:1.6}. 
\section{Preliminaries}
In this section we recall some properties of the Dirichlet heat kernel $\Gamma_D$ 
on the exterior $\Omega$ of the ball $B(0,1)$
and obtain some estimates of integral operators $S_1(t)$, $S_2(t)$ and $F$.
\vspace{3pt}

We first recall some properties of the semigroup $S_1(t)$.  
By  \cite[Theorem~16.3]{LSU} (see also \cite{GS, IK}) 
we find $C>0$ such that 
$$
|\nabla_x^j \Gamma_D(x,y,t)|\le Ct^{-\frac{N}{2}}h(t)^j\exp\left(-C\frac{|x-y|^2}{t}\right),
\qquad x,y\in\Omega,\quad  t>0,
$$
where 
\begin{equation}\label{eq:2.1}
h(t):=\max\{1,t^{-1/2}\}
\end{equation}
and $j\in\set{0,1}$.
Then we have: 
\begin{itemize}
\item[(${\rm G_1}$)]
There exists $c_1=c_1(N)$ such that 
$$
\|\nabla^jS_1(t)\phi\|_{L^p}\le c_1 t^{-\frac{N}{2}\left(\frac{1}{q}-\frac{1}{p}\right)}h(t)^j\|\phi\|_{L^q},
\qquad t>0,
$$
for $\phi\in L^{q,\infty}(\Omega)$,  $1\le p\le q\le\infty$ and $j\in\set{0,1}$;
\item[(${\rm G_2}$)]
Let $0\le\gamma<N$. Assume that a measurable function $f$ in $\Omega$ satisfies 
$$
|f(x)|\le|x|^{-\gamma}
$$
for almost all $x\in\Omega$. Then there exists $c_2=c_2(N,\gamma)>0$
such that 
$$
\|\nabla^jS_1(t)f\|_\infty\le c_2(1+t)^{-\frac{\gamma}{2}}h(t)^j,\qquad t>0,
$$
where $j\in\set{0,1}$
(see e.g. \cite{FKS});
\item[(${\rm G_3}$)]
Let $\phi\in L^q(\Omega)$ with $1\le q\le\infty$ and $\tau>0$. 
Then $S_1(t)\phi$ is bounded and smooth in $\overline{\Omega}\times(\tau,\infty)$.
\end{itemize} 
\vspace{5pt}

Next we recall some properties of the kernel ${\mathcal K}$ and $S_2(t)\psi$. 
By \cite[Lemmata~2.1 and~2.2]{FIK01} we have the following two lemmata.

\begin{lemma}
\label{Lemma:2.1}
Let $N\ge 3$ and ${\mathcal K}$ be as in \eqref{eq:1.5}. 
Then 
\begin{equation}
\label{eq:2.2}
 \int_{\partial\Omega}{\mathcal K}(x,y,t)\,d\sigma_y
 = \int_{\partial\Omega}K(e^tx,y)\,d\sigma_y
=(e^t|x|)^{-(N-2)}
\end{equation}
for $(x,t)\in\overline{\Omega}\times[0,\infty)$ with $e^tx\in\Omega$. 
\end{lemma}
\begin{lemma}
\label{Lemma:2.2}
Let $N\ge 3$ and
$\psi$ be a nonnegative measurable function on $\partial\Omega$ such that 
$\psi\in L^\infty(\partial\Omega)$.
Then 
\begin{align}
\label{eq:2.3}
 & 
 S_2(\cdot)\psi \in C^\infty(\overline{\Omega}\times(0,\infty))\cap C^\infty(\Omega\times[0,\infty)),\\
\notag 
 &  -\Delta_x S_2(t)\psi=0\quad\mbox{in}\quad\Omega\quad\mbox{for any $t\ge 0$},\\
\notag
 &  S_2(t)\left[S_2(s)\psi\right]^b=S_2(t+s)\psi
 \quad\mbox{for $s>0$ and $t\ge 0$},\\
\label{eq:2.4}
 &  |[S_2(t)\psi](x)|\le e^{-(N-2)t}|x|^{-(N-2)}|\psi|_\infty\quad\mbox{in}\quad\overline{\Omega}\times[0,\infty). 
\end{align}
Here $[S_2(t)\psi]^b$ is the restriction of $S_2(t)\psi$ to $\partial\Omega$. 
Furthermore, for any $\theta\in(0,1)$, 
there exists $c_3>0$ such that for every nonnegative $\psi \in L^\infty(\partial\Omega)$
\begin{equation}
\label{eq:2.5}
 \|S_2(t)\psi\|_{C^{1,\theta}(\overline{\Omega})}\le c_3|\psi|_{C^{1,\theta}},
 \qquad t\ge 0.
\end{equation}
\end{lemma}
Then we have: 
\begin{lemma}
\label{Lemma:2.3}
Let $N\ge3$ and $\psi\in C^{1,\theta}(\partial\Omega)$ with $\theta\in(0,1)$.
Let $F_1[\psi]$ be as in \eqref{eq:1.13}.
Then 
\begin{equation}
\label{eq:2.6}
|F_1[\psi](x,t)|\le c_3|x||\psi|_{C^{1,\theta}},\qquad x\in\overline\Omega,\quad t>0,
\end{equation}
and
\begin{equation}
\label{eq:2.7}
|F_1[\psi](x,t)|\le N(e^t|x|)^{-(N-2)}\frac{e^t|x|}{e^t|x|-1}|\psi|_{L^{\infty}},\qquad x\in\overline\Omega,\quad t>0,
\end{equation}
where $c_3$ is the constant given in \eqref{eq:2.5}.
\end{lemma}
\begin{proof}
Let $x\in\overline\Omega$ and $t>0$. 
We prove \eqref{eq:2.6}. 
It follows from \eqref{eq:1.5} that
\begin{equation}
\label{eq:2.8}
\begin{split}
\partial_t\mathcal K(x,y,t)
=e^tx\cdot(\nabla_x K)(e^t x,y)=x\cdot\nabla_x\mathcal K(x,y,t)\quad \text{for }\; y\in\partial\Omega.
\end{split}
\end{equation}
Then, 
by \eqref{eq:1.13}, \eqref{eq:2.8} and \eqref{eq:1.6} we have
\begin{equation}
\label{eq:2.9}
\begin{split}
F_1[\psi](x,t)
&
=\int_{\partial\Omega}\partial_t\mathcal K(x,y,t)\psi(y)\,d\sigma_y
\\
&
=\int_{\partial\Omega}x\cdot\nabla_x\mathcal K(x,y,t)\psi(y)\,d\sigma_y
=x\cdot \nabla_x[S_2(t)\psi](x).
\end{split}
\end{equation}
This together with \eqref{eq:2.5} implies \eqref{eq:2.6}.

We prove \eqref{eq:2.7}. 
Let $i$, $j\in\{1,\dots,N\}$ and $y\in\partial\Omega$. 
By \eqref{eq:1.4} we have
\begin{equation}
\label{eq:2.10}
\begin{split}
\partial_{x_i}K(x,y)
&
=(2-N)|x|^{-N}x_i P\left(\frac{x}{|x|^2},y\right)+|x|^{-(N-2)}\partial_{x_i}P\left(\frac{x}{|x|^2},y\right)
\\
&
=(2-N)\frac{x_i}{|x|^2}K(x,y)+|x|^{-(N-2)}\sum_{j=1}^N\frac{\partial z_j}{\partial x_i}\partial_{z_j}P\left(z,y\right),
\end{split}
\end{equation}
where $z=x/|x|^2$. 
Since 
\begin{equation*}
\begin{split}
 & \frac{\partial z_j}{\partial x_i}=\delta_{ij}|z|^2-2z_iz_j,\\
 & \partial_{z_j}P(z,y)=-c_N|z-y|^{-N-2}\bigg(2z_j|z-y|^2+N(1-|z|^2)(z_j-y_j)\bigg),
\end{split}
\end{equation*}
we obtain 
\begin{equation}
\label{eq:2.11}
\begin{split}
&
\sum_{j=1}^N\frac{\partial z_j}{\partial x_i}\partial_{z_j}P\left(z,y\right)
\\
&
=-c_N|z-y|^{-N-2}\sum_{j=1}^N(\delta_{ij}|z|^2-2z_iz_j)\bigg(2z_j|z-y|^2+N(1-|z|^2)(z_j-y_j)\bigg)
\\
&
=c_N|z-y|^{-N-2}\bigg[2z_i|z|^2|z-y|^2+N(1-|z|^2)(|z|^2z_i+|z|^2y_i-2z_i(z\cdot y))\bigg]
\\
&
=P(z,y)\bigg[2z_i|z|^2(1-|z|^2)^{-1}+N|z-y|^{-2}(|z|^2z_i+|z|^2y_i-2z_i(z\cdot y))\bigg].
\end{split}
\end{equation}
Since $x\cdot z=1$, $|x|^{-1}=|z|$, $|x|^2(z\cdot y)=2(x\cdot y)$ and $|y|=1$, 
by \eqref{eq:1.4}, \eqref{eq:2.10} and \eqref{eq:2.11} we see that  
\begin{equation}
\label{eq:2.12}
\begin{split}
 & x\cdot\nabla K(x,y)
 =(2-N)K(x,y)+|x|^{-(N-2)}P(z,y)\\
 & \qquad\times\biggr[2(x\cdot z)\frac{|z|^2}{1-|z|^2}+N|z-y|^{-2}(|z|^2(x\cdot z)+|z|^2(x\cdot y)-2(x\cdot z)(z\cdot y))\biggr]\\
 & =K(x,y)\left[2-N+\frac{2}{|x|^2-1}+\frac{N}{|z-y|^2|x|^2}(1-x\cdot y)\right]\\
 & =K(x,y)\left[2-N+\frac{2}{|x|^2-1}+\frac{N}{|x-y|^2}(1-x\cdot y)\right].
\end{split}
\end{equation}
Since
$$
\frac{1}{2}\ge \frac{1-x\cdot y}{|x-y|^2}
\ge -\frac{|x|-1}{|x|^2-2x\cdot y+1}
\ge -\frac{|x|-1}{|x|^2-2|x|+1}=-\frac{1}{|x|-1},
$$
it follows from \eqref{eq:2.12} 
that
\begin{equation}
\label{eq:2.13}
|x\cdot\nabla K(x,y)|\le N\frac{|x|}{|x|-1}K(x,y). 
\end{equation}
This together with \eqref{eq:2.9} and \eqref{eq:2.2} implies 
\begin{align*}
|F_1[\psi](x,t)|
 & \le\int_{\partial\Omega} e^tx\cdot(\nabla_x K)(e^t x,y)|\psi(y)|\,d\sigma_y\\
 & \le N\frac{e^t|x|}{e^t|x|-1}|\psi|_{L^\infty}\int_{\partial\Omega}K(e^tx,y,t)\, d\sigma_y
\le N (e^t|x|)^{-(N-2)}\frac{e^t|x|}{e^t|x|-1}|\psi|_{L^{\infty}}.
\end{align*}
Thus \eqref{eq:2.7} holds, and the proof of Lemma~\ref{Lemma:2.3} is complete.
\end{proof}
By Lemma~\ref{Lemma:2.3} we obtain the following lemma. 
\begin{lemma}
\label{Lemma:2.4}
Let $N\ge3$ and $\theta\in(0,1)$. For $\psi \in C^{1,\theta}(\partial\Omega)$ set
\begin{equation}
\label{eq:2.14}
D_\varepsilon[\psi](x,t):=\int_0^t[S_1(\varepsilon^{-1}(t-s))F_1[\psi](s)](x)\,ds
\end{equation}
for $x\in\overline\Omega$, $t>0$ and $\varepsilon>0$.
Then $D_\varepsilon[\psi]$ and $\nabla D_\varepsilon[\psi]$ are bounded 
and smooth in $\overline\Omega\times(\tau,\infty)$ for $\tau>0$.
Furthermore, there exists $C>0$ such that 
for every $\psi \in C^{1,\theta}(\partial\Omega)$
\begin{eqnarray}
\label{eq:2.15}
 & & \|D_\varepsilon[\psi](t)\|_{L^\infty}
 \le C\varepsilon^{\frac{\alpha}{2}}t^{\frac{2-\alpha}{2}}|\psi|_{C^{1,\theta}},\\
\label{eq:2.16}
 & & \|\nabla D_\varepsilon[\psi](t)\|_{L^\infty}
 \le C\varepsilon^{\frac{\alpha}{2}}t^{\frac{2-\alpha}{2}}|\psi|_{C^{1,\theta}},
\end{eqnarray}
for $t>0$ and $\varepsilon>0$,
where $\alpha$ is as in \eqref{eq:1.17}.
\end{lemma}
\begin{proof}
We first prove \eqref{eq:2.15}.
By \eqref{eq:2.6} there is $C_1>0$ such that 
\begin{equation}
\label{eq:2.17}
|F_1[\psi](y,s)|\le \f{C_1}{4}|\psi|_{C^{1,\theta}}|y|
\le C_1|\psi|_{C^{1,\theta}}|y|^{-\alpha}
\le C_1|\psi|_{C^{1,\theta}}|y|^{-(\alpha-1)}
\end{equation}
for $y\in\Omega$ with $1\le |y|\le 2$, $s>0$  
and every $\psi \in C^{1,\theta}(\partial\Omega)$. 
Since 
$$
e^s|y|-1\ge\frac{1}{2}e^s|y|\quad\mbox{for}\quad |y|\ge 2,\,\,s>0, 
$$ 
by \eqref{eq:2.7} and \eqref{eq:1.17} we obtain
\begin{equation}
\label{eq:2.18}
|F_1[\psi](y,s)|\le 2N|\psi|_{L^\infty} |y|^{-(N-2)}\le 2N|\psi|_{L^\infty} |y|^{-\alpha}
\le 2N|\psi|_{L^\infty} |y|^{-(\alpha-1)}
\end{equation}
for every $y\in\Omega$ with $|y|\ge 2$, $s>0$ and $\psi \in C^{1,\theta}(\partial\Omega)$ 
Since $0\le\alpha<N$, 
by \eqref{eq:2.17} and \eqref{eq:2.18} 
we apply property~$(G_2)$ with $\gamma=\alpha$ and $j=0$ to obtain $C_2>0$ such that 
\begin{equation}
\label{eq:2.19}
|D_\varepsilon[\psi](x,t)|\le C_2(|\psi|_{C^{1,\theta}}+|\psi|_{L^\infty})
\int_0^t \tau_\varepsilon^{-\frac{\alpha}{2}}\,ds
\le \f{2C_2}{1-\f{\alpha }2}|\psi|_{C^{1,\theta}}\varepsilon^{\frac{\alpha}{2}}t^{\frac{2-\alpha}{2}}
\end{equation}
for 
$\psi \in C^{1,\theta}(\partial\Omega)$, $x\in\overline\Omega$, $t>0$ and $\varepsilon>0$, 
where $\tau_\varepsilon:=\varepsilon^{-1}(t-s)$. Thus \eqref{eq:2.15} holds. 
Furthermore, 
since $h(\tau_\varepsilon)=1$ for $t> \varepsilon$ and $s\in(0,t-\varepsilon)$, 
similarly to \eqref{eq:2.19}, 
by property~$(G_2)$ with $\gamma=\alpha$ and $j=1$ we have 
\begin{equation}
\label{eq:2.20}
\begin{split}
 & \biggr|\int_0^{t-\varepsilon}
\nabla [S_1(\tau_\varepsilon)F_1[\psi](s)](x)\,ds\biggr|
  \le C_2(|\psi|_{C^{1,\theta}}+|\psi|_{L^\infty})
\int_0^{t-\epsilon} \tau_\varepsilon^{-\frac{\alpha}{2}}\,ds
\le \f{2C_2}{1-\frac{\alpha}2}|\psi|_{C^{1,\theta}}\varepsilon^{\frac{\alpha}{2}}t^{\frac{2-\alpha}{2}}
\end{split}
\end{equation}
for every $\psi \in C^{1,\theta}(\partial\Omega)$, $x\in\overline\Omega$ and $t\ge\varepsilon>0$.
On the other hand, 
since $0\le\alpha-1<N$ 
and $h(\tau_\varepsilon)=\tau_\varepsilon^{-1/2}$ for $s\in[\max\{0,t-\varepsilon\},t]$, 
by \eqref{eq:2.17} and \eqref{eq:2.18} 
we apply property~$(G_2)$ with $\gamma=\alpha-1$ and $j=1$ to obtain $C_3>0$ such that 
\begin{equation}
\label{eq:2.21}
\begin{split}
 & \biggr|\int_{\max\{0,t-\varepsilon\}}^t
\nabla [S_1(\tau_\varepsilon)F_1[\psi](s)](x)\,ds\biggr|\\
 & \le C_3(|\psi|_{C^{1,\theta}}+|\psi|_{L^\infty})
\int_{\max\{0,t-\varepsilon\}}^t \tau_\varepsilon^{-\frac{1}{2}}\tau_\varepsilon^{-\frac{\alpha-1}{2}}\,ds
\le C_3|\psi|_{C^{1,\theta}}\varepsilon^{\frac{\alpha}{2}}t^{\frac{2-\alpha}{2}}
\end{split}
\end{equation}
for 
every $\psi \in C^{1,\theta}(\partial\Omega)$, $x\in\overline\Omega$, $t>0$ and $\varepsilon>0$. 
By \eqref{eq:2.20} and \eqref{eq:2.21} we have \eqref{eq:2.16}. 

We now fix $\psi \in C^{1,\theta}(\partial\Omega)$ and $\varepsilon >0$. 
It remains to prove that 
$D_\varepsilon[\psi]$ and $\nabla D_\varepsilon[\psi]$ are bounded and smooth 
in $\overline\Omega\times(\tau,\infty)$ for $\tau>0$.
It follows from the semigroup property of $S_1(t)$ that 
\begin{equation*}
\begin{split}
 & D_\varepsilon[\psi](x,t)
 =\int_0^t[S_1(\varepsilon^{-1}(t-s))F_1[\psi](s)](x)\,ds\\
 & =S_1(\varepsilon^{-1}(t-\tau/2))D_\varepsilon[\psi](x,\tau/2)
+\int_{\tau/2}^t[S_1(\varepsilon^{-1}(t-s))F_1[\psi](s)](x)\,ds
\end{split}
\end{equation*}
for $x\in\overline\Omega$ and $0<\tau<t<\infty$. 
We observe from \eqref{eq:2.15} and (${\rm G_3}$) that 
$$
(x,t)\mapsto S_1(\varepsilon^{-1}(t-\tau/2))D_\varepsilon[\psi](x,\tau/2)
$$
is bounded and smooth in $\overline\Omega\times(\tau,\infty)$.
On the other hand it holds from \eqref{eq:1.13} that 
$$
(x,t)\mapsto F_1[\psi](x,t)=\partial_t[S_2(t)\psi](x).
$$
Then, by \eqref{eq:2.3}
we apply the same argument as in \cite[Section~16, Chapter~4]{LSU}
to see that 
$$
\int_{\tau/2}^t[S_1(\varepsilon^{-1}(t-s))F_1[\psi](s)](x)\,ds
$$
is bounded and smooth in $\overline\Omega\times(\tau,\infty)$.
Therefore we deduce that 
$D_\varepsilon[\psi]$ and $\nabla D_\varepsilon[\psi]$ are bounded and smooth 
in $\overline\Omega\times(\tau,\infty)$.
Thus Lemma~\ref{Lemma:2.4} follows.  
\end{proof}
%
\section{Proof of Theorem~\ref{Theorem:1.5}}
We introduce some notation. 
Let $T\in(0,\infty)$, $\varepsilon\in(0,1)$ and $\alpha$ be as in \eqref{eq:1.17}.
Let $L>0$. Set 
$$
X_{T,L}:=\bigg\{v\mid v, \nabla v\in C(\overline\Omega\times(0,T))\,:\,\|v\|_{X_{T,L}}<\infty\bigg\},
\quad
\|v\|_{X_{T,L}}:=\sup_{0<t<T}e^{-Lt}E_\varepsilon[v](t),
$$
where 
$$
E_\varepsilon[v](t):=\left(1+(\varepsilon^{-1}t)^{\frac{\alpha}{2}}\right)\|v(t)\|_{L^\infty}
+(\varepsilon^{-1}t)^{\frac{1}{2}}\left(1+(\varepsilon^{-1}t)^{\frac{\alpha-1}{2}}\right)\|\nabla v(t)\|_{L^\infty}.
$$
Then $X_{T,L}$ is a Banach space equipped with the norm $\|\cdot\|_{X_{T,L}}$. 
For the proof of assertion~(i) of Theorem~\ref{Theorem:1.5},
we will apply the contraction mapping theorem in $X_{T,L}$ 
to find a fixed point of 
\begin{equation}
\label{eq:3.1}
Q_\varepsilon[v](t)
:=S_1(\varepsilon^{-1}t)\Phi
-D_\varepsilon[\varphi_b](t)+\int_0^tS_1(\varepsilon^{-1}(t-s))F_2[v](s)\,ds,
\end{equation}
where $\Phi$, $F_2[v]$ and $D_\varepsilon[\varphi_b]$ are as in \eqref{eq:1.10}, \eqref{eq:1.14} and \eqref{eq:2.14}, respectively. 
To this end, we prepare two lemmata. 
\begin{lemma}
\label{Lemma:3.1}Let $N\ge3$ and $\beta\in(0,1)$.
There exists $C>0$ such that for every $T\in(0,\infty)$ and $L>0$,
\begin{equation}
\label{eq:3.2}
F_2[v](x,t)\le C(\varepsilon^{-1}t)^{-\frac{1}{2}}e^{Lt}|x|^{-(N-2)}\bigg\{1+|x|\bigg(\frac{t}{|x|-1}\bigg)^\beta\bigg\}\|v\|_{X_{T,L}}
\end{equation}
for $x\in\Omega$, $0<t<T$, $\varepsilon\in(0,1)$ and $v\in X_{T,L}$. 
\end{lemma}
\noindent
\begin{proof}
Let $T>0$, $\varepsilon\in(0,1)$ and $v\in X_{T,L}$. 
It follows from \eqref{eq:1.14} that 
\begin{equation}
\label{eq:3.3}
F_2[v](x,t)=F_2'[v](x,t)+F_2''[v](x,t)
\end{equation}
for $x\in\Omega$ and $0<t<T$, 
where 
\begin{align*}
 & F_2'[v](x,t):=\int_{\partial\Omega}K(x,y)\partial_{\nu}v(y,t)\, d\sigma_y,\\
 & F_2''[v](x,t):=\int_0^t\int_{\partial\Omega}\partial_t \mathcal K(x,y,t-s)\partial_{\nu}v(y,s)\, d\sigma_y\,ds.
\end{align*}
Since 
\begin{equation}
\label{eq:3.4}
\sup_{0<t<T}\,e^{-Lt}(\varepsilon^{-1}t)^{\frac{1}{2}}\|\nabla v(t)\|_{L^\infty}\le\|v\|_{X_{T,L}}, 
\end{equation}
by \eqref{eq:2.2} we see that 
\begin{equation}
\label{eq:3.5}
\begin{split}
|F_2'[v](x,t)|
&\le \int_{\partial\Omega}K(x,y)|\partial_{\nu}v(y,t)|\, d\sigma_y\\
&
 \le \|\nabla v(t)\|_{L^\infty}|x|^{-(N-2)}
\le (\varepsilon^{-1}t)^{-\frac{1}{2}}e^{Lt}|x|^{-(N-2)}\|v\|_{X_{T,L}}
\end{split}
\end{equation}
for $x\in\Omega$ and $t>0$. 
On the other hand,
since
$$
e^t|x|-1\ge e^t-1\ge t,\qquad
e^t|x|-1\ge |x|-1,
$$
for $x\in\Omega$ and $t>0$, 
for any $\beta\in(0,1)$
it follows from \eqref{eq:2.8} and \eqref{eq:2.13} that
$$
|\partial_t \mathcal K(x,y,t)|
\le N\frac{e^t|x|}{e^t|x|-1}\mathcal K(x,y,t)
\le N\frac{e^t|x|}{t^{1-\beta}(|x|-1)^\beta}\mathcal K(x,y,t)
$$
for $x\in\Omega$ and $0<t<T$. 
Then, by \eqref{eq:2.2} and \eqref{eq:3.4} we obtain
\begin{equation}
\label{eq:3.6}
\begin{split}
|F_2''[v](x,t)|
& \le N\frac{|x|}{(|x|-1)^\beta}\int_0^t \frac{e^{t-s}}{(t-s)^{1-\beta}}
\int_{\partial\Omega}\mathcal K(x,y,t-s)\|\nabla v(s)\|_{L^\infty}\,d\sigma_y\,ds\\
&  =N\frac{|x|}{(|x|-1)^\beta}\int_0^t \frac{e^{t-s}}{(t-s)^{1-\beta}}(e^{t-s}|x|)^{-(N-2)}\|\nabla v(s)\|_{L^\infty}\,ds\\
& \le N\frac{|x|^{-(N-3)}}{(|x|-1)^\beta}\|v\|_{X_{T,L}}
\int_0^t (t-s)^{-1+\beta}(\varepsilon^{-1}s)^{-\frac{1}{2}}e^{Ls}\,ds\\
&
\le C\varepsilon^{\frac{1}{2}}\frac{|x|^{-(N-3)}}{(|x|-1)^\beta}t^{-\frac{1}{2}+\beta}e^{Lt}\|v\|_{X_{T,L}}
\\
&
=  C(\varepsilon^{-1}t)^{-\frac{1}{2}}e^{Lt}|x|^{-(N-3)}\bigg(\frac{t}{|x|-1}\bigg)^\beta\|v\|_{X_{T,L}}
\end{split}
\end{equation}
for $x\in\Omega$ and $0<t<T$ and $C=N\int_0^1(1-\sigma)^{\beta-1}\sigma^{-\f12}d\sigma$. 
Therefore, by \eqref{eq:3.3}, \eqref{eq:3.5} and \eqref{eq:3.6} we have \eqref{eq:3.2}.
Thus Lemma~\ref{Lemma:3.1} follows. 
\end{proof}
\begin{lemma}
\label{Lemma:3.2}Let $N\ge 3$. 
For any $T\in(0,\infty)$, $L>0$, $v\in X_{T,L}$ and $\varepsilon\in(0,1)$,
set
\begin{equation}
\label{eq:3.7}
\tilde{D}_\varepsilon[v](t):=
\int_0^tS_1(\varepsilon^{-1}(t-s))F_2[v](s)\,ds.
\end{equation}
Then for every $T>0$ there exists $L_*>0$ such that 
\begin{equation}
\label{eq:3.8}
\|\tilde{D}_\varepsilon[v]\|_{X_{T,L}}\le\frac{1}{2}\|v\|_{X_{T,L}}
\end{equation}
for $v\in X_{T,L}$, $\varepsilon\in(0,1)$ and $L\ge L_*$.
Furthermore,
for any $0<\tau<T$ and $j\in\{0,1\}$,
$$
\nabla^j\tilde{D}_\varepsilon[v]\in C^\infty(\Omega\times(\tau,T))\cap BC^1(\overline\Omega\times(\tau,T)).
$$
\end{lemma}
For the proof of Lemma~\ref{Lemma:3.2} we prepare the following lemma. 
\begin{lemma}
\label{Lemma:3.3}
Let $0\le a<1$ and $0\le b<1$ be such that $0\le a+b<1$. 
Let $\gamma\ge 0$ and $T>0$. 
Then, for any $\delta>0$, 
there exists $L_*\ge 1$ such that 
$$
\sup_{0<t<T}e^{-Lt}t^\gamma\int_0^t e^{Ls}s^{-a}(t-s)^{-b}\,ds\le\delta\quad\mbox{for}\quad L\ge L_*.
$$
\end{lemma}
\begin{proof}
Let $T>0$, $\gamma\ge0$ and $\delta>0$. 
For any $\mu\in(0,1)$ and $L>0$, we have
\begin{equation*}
\begin{split}
 & 
 \int_0^t e^{Ls}s^{-a}(t-s)^{-b}\,ds
 \\
 & 
 =\bigg(\int_0^{\mu t}+\int_{\mu t}^{(1-\mu)t}+\int_{(1-\mu)t}^t\bigg) e^{Ls}s^{-a}(t-s)^{-b}\,ds
 \\
 & 
 \le (1-\mu)^{-b}t^{-b}e^{Lt}\int_0^{\mu t}s^{-a}\,ds+\mu^{-a-b}t^{-a-b}\int_0^t e^{Ls}\,ds
 +(1-\mu)^{-a}t^{-a}e^{Lt}\int_{(1-\mu)t}^t(t-s)^{-b}\,ds
 \\
 & 
 =
 \frac{1}{1-a}(1-\mu)^{-b}\mu^{1-a}t^{1-a-b}e^{Lt}
 +\frac{\mu^{-a-b}t^{-a-b}}{L}(e^{Lt}-1)
 +\frac{1}{1-b}(1-\mu)^{-a}\mu^{1-b}t^{1-a-b}e^{Lt}
\end{split}
\end{equation*}
for $0<t<T$ and $L>0$. 
Then, since $a+b<1$ and $\gamma\ge0$, taking a sufficiently small $\mu\in(0,1/2)$ if necessary,
we obtain 
\begin{equation}
\label{eq:3.9}
\begin{split}
 e^{-Lt}t^\gamma\int_0^t e^{Ls}s^{-a}(t-s)^{-b}\,ds
 & 
 \le C(\mu^{1-a}+\mu^{1-b})t^{1-a-b+\gamma}+\mu^{-a-b}t^{-a-b+\gamma}\frac{1-e^{-Lt}}{L}
 \\
 & 
 \le C(\mu^{1-a}+\mu^{1-b})T^{1-a-b+\gamma}+\mu^{-a-b}t^{-a-b+\gamma}\frac{1-e^{-Lt}}{L}
 \\
 &
 \le \frac{\delta}{2}+\mu^{-a-b}t^{-a-b+\gamma}\frac{1-e^{-Lt}}{L}
\end{split}
\end{equation}
for $0<t<T$ and $L>0$, where $C=2^b(1-a)^{-1}+2^a(1-b)^{-1}$.
Let 
\[
 f(t,L):=t^{\gamma-a-b}\f{1-e^{-Lt}}{L}, \qquad t\in(0,T),\, L>0,
\]
Then we see that in the case of $\gamma\ge a+b$, 
$0\le f(t,L)\le T^{\gamma-a-b}L^{-1}$ for all $t\in(0,T)$ and $L>0$, 
and the choice $L_*=2\mu^{-a-b}\delta^{-1}T^{\gamma-a-b}$ verifies the lemma. 
If, on the other hand, $\gamma<a+b$, then for every $L>0$, $\limsup_{t\to0^+} f(t,L)=0$, and thus 
\[
 t_L:=\argmax\displaylimits_{t\in(0,T)} f(t,L)
\]
exists and satisfies $t_L\in(0,T)$ and 
\[
 \f{d}{dt} f(t,L)\vert_{t=t_L} = 0\; \text{ and hence }\;
 1-e^{-Lt_L} = \f{L}{a+b-\gamma}t_Le^{-Lt_L},
\]
so that 
\[
 f(t,L)\le f(t_L,L)=\frac{t_L^{\gamma+1-a-b}}{a+b-\gamma}e^{-Lt_L}\quad \text{for all } t\in(0,T), L>0.
\]
As, for any $L>0$,  
\[
 \sup_{s>0} s^{\gamma+1-a-b}e^{-Ls} = L^{-\gamma-(1-a-b)}(\gamma+1-a-b)^{\gamma+1-a-b}e^{-(\gamma+1-a-b)}, 
\]
we may conclude that also in the case $\gamma<a+b$, $\sup_{t\in(0,T)} f(t,L)\to 0$ as $L\to\infty$, which together with \eqref{eq:3.9} completes the proof of Lemma~\ref{Lemma:3.3}.
\end{proof}

We prove Lemma~\ref{Lemma:3.2}.
\begin{proof}[\textit{\textbf{Proof of Lemma~\ref{Lemma:3.2}}}]
Let $0<T<\infty$ and let $\alpha$ be as in \eqref{eq:1.17}. 
Let
\begin{equation}
\label{eq:3.10}
0<\beta<\min\left\{\frac{\alpha-1}{N},2-\alpha\right\} \quad\mbox{if}\quad N\ge 4,
\qquad
0<\beta<\frac{1}{4}\quad\mbox{if}\quad N=3. 
\end{equation}
It follows from \eqref{eq:3.2} and \eqref{eq:1.17} that with  $C_1>0$ as in \eqref{eq:3.2} 
\begin{equation}
\label{eq:3.11}
\begin{split}
|F_2[v](y,s)|
 & \le C_1e^{Ls}(\varepsilon^{-1}s)^{-\frac{1}{2}}\|v\|_{X_{T,L}}
 \biggr[|y|^{-(N-2)}+s^\beta \eta_\beta(y)+s^\beta|y|^{-(N-3+\beta)}\chi_{\{|y|>2\}}\biggr]\\
 & \le C_1e^{Ls}\varepsilon^{\frac{1}{2}}s^{-\frac{1}{2}}\|v\|_{X_{T,L}}
 \biggr[|y|^{-(\alpha-1)}+s^\beta\eta_\beta(y)+
s^\beta|y|^{-(\alpha-1+\beta)}\chi_{\{|y|>2\}}\biggr]
\end{split}
\end{equation}
for $y\in\Omega$, $L>0$ and $0<s<T$, where $\eta_\beta(y):=(|y|-1)^{-\beta}\chi_{\{1\le|y|\le 2\}}$. 
By \eqref{eq:3.10} we have
\begin{equation}
\label{eq:3.12}
\eta_\beta\in L^{\frac{N}{\alpha-1}}(\Omega)\quad\mbox{if}\quad N\ge 4,
\qquad
\eta_\beta\in L^4(\Omega)\quad\mbox{if}\quad N\ge 3. 
\end{equation}

For $s>0$ and $\tau>0$, set
\begin{equation}
\label{eq:3.13}
 I(s,\tau):=
 s^{-\frac{1}{2}}
 \left[\tau^{-\frac{\alpha-1}{2}}
 +s^\beta\tau^{-d}+s^\beta \tau^{-\frac{\alpha-1+\beta}{2}}\right]
 \end{equation}
where $d=(\alpha-1)/2$ if $N\ge 4$ and $d=3/8$ if $N=3$. 
Then it follows that 
\begin{equation}
\label{eq:3.14}
I(s,\varepsilon^{-1}\tau)\le \varepsilon^{\frac{\alpha-1}{2}}I(s,\tau)
\quad\mbox{for}\quad s>0,\,\,\tau>0,\,\,\varepsilon \in(0,1).
\end{equation}
By \eqref{eq:3.11} and \eqref{eq:3.12} 
we apply properties~$(G_2)$ and $(G_1)$ to obtain $C_2>0$ such that
\begin{equation}
\label{eq:3.15}
\begin{split}
 & \|\nabla^j S_1(\tau_\varepsilon)F_2[v](s)\|_{L^\infty}\\
 & \le C_2e^{Ls}\varepsilon^{\frac{1}{2}}s^{-\frac{1}{2}}\|v\|_{X_{T,L}}h(\tau_\varepsilon)^j
\left[(1+\tau_\varepsilon)^{-\frac{\alpha}{2}}
+s^\beta\tau_\varepsilon^{-d}+s^\beta(1+\tau_\varepsilon)^{-\frac{\alpha-1+\beta}{2}}\right]\\
 & \le C_2e^{Ls}\varepsilon^{\frac{1}{2}}\|v\|_{X_{T,L}}
 \times\left\{
 \begin{array}{ll}
  I(s,\tau_\epsilon) & \mbox{if}\quad j=0 \quad \mbox{or}\quad j=1,\,\,\tau_\varepsilon\ge 1,\vspace{3pt}\\
  \tau_\epsilon^{-\frac{1}{2}}I(s,\tau_\epsilon) & \mbox{if}\quad j=1,\,\,0<\tau_\varepsilon<1,
\end{array}
\right.
\end{split}
\end{equation}
for $L>0$, $0<s<t$, where $\tau_\varepsilon=\varepsilon^{-1}(t-s)$ and $h$ is as in \eqref{eq:2.1}.
By \eqref{eq:3.7}, \eqref{eq:3.15} and \eqref{eq:3.14} we have 
\begin{equation}
\label{eq:3.16}
\begin{split}
\|\tilde{D}_\varepsilon[v](t)\|_{L^\infty} 
& \le\int_0^t\|S_1(\tau_\varepsilon)F_2[v](s)\|_{L^\infty}\,ds\\
 & \le C_2\varepsilon^{\frac{1}{2}}\|v\|_{X_{T,L}}\int_0^t 
 e^{Ls}I(s,\tau_\varepsilon)\,ds\le C_2\varepsilon^{\frac{\alpha}{2}}\|v\|_{X_{T,L}}\int_0^t e^{Ls}I(s,t-s)\,ds
\end{split}
\end{equation}
for $t>0$, $L>0$ and $\varepsilon\in(0,1)$.
This implies that
\begin{equation*}
\begin{split}
  e^{-Lt}\left(1+(\varepsilon^{-1}t)^{\frac{\alpha}{2}}\right)\|\tilde{D}_\varepsilon[v](t)\|_{L^\infty}
 & \le C_2e^{-Lt}\left(1+(\varepsilon^{-1}t)^{\frac{\alpha}{2}}\right)\varepsilon^{\frac{\alpha}{2}}\|v\|_{X_{T,L}}\int_0^t e^{Ls}I(s,t-s)\,ds\\
 & \le C_2\|v\|_{X_{T,L}}e^{-Lt}\left(1+t^{\frac{\alpha}{2}}\right)\int_0^t e^{Ls}I(s,t-s)\,ds
\end{split}
\end{equation*}
for $t>0$, $L>0$ and $\varepsilon\in(0,1)$. 
Then, by Lemma~\ref{Lemma:3.3} with \eqref{eq:3.13}, 
taking a sufficiently large $L\ge 1$ if necessary, we obtain 
\begin{equation}
\label{eq:3.17}
\sup_{0<t<T}
e^{-Lt}\left(1+(\varepsilon^{-1}t)^{\frac{\alpha}{2}}\right)\|\tilde{D}_\varepsilon[v](t)\|_{L^\infty}
\le\frac{1}{4}\|v\|_{X_{T,L}}
\end{equation}
for $0<\varepsilon<1$. 
Similarly to \eqref{eq:3.16}, it follows from \eqref{eq:3.14} and \eqref{eq:3.15} that
\begin{equation*}
\begin{split}
 & \|\nabla\tilde{D}_\varepsilon[v](t)\|_{L^\infty}
 \le\biggr(\int_0^{\max\{t-\varepsilon,0\}}+\int_{\max\{t-\varepsilon,0\}}^t\biggr)
 \|\nabla S_1(\tau_\varepsilon)F_2[v](s)\|_{L^\infty}\,ds\\
 & \le C_2\varepsilon^{\frac{1}{2}}\|v\|_{X_{T,L}}
 \left\{\int_0^{\max\{t-\varepsilon,0\}} e^{Ls}I(s,\tau_\varepsilon)\,ds
 +\int_{\max\{t-\varepsilon,0\}}^t e^{Ls}\tau_\varepsilon^{-\frac{1}{2}}I(s,\tau_\varepsilon)\,ds\right\}\\
 & \le C_2\|v\|_{X_{T,L}}
 \left\{\varepsilon^{\frac{\alpha}{2}}\int_0^t e^{Ls}I(s,t-s)\,ds
 +\varepsilon^{\frac{\alpha+1}{2}}\int_0^t e^{Ls}(t-s)^{-\frac{1}{2}}I(s,t-s)\,ds\right\}
\end{split}
\end{equation*}
for $0<t<T$ and $L>0$. 
Then we have 
\begin{equation*}
\begin{split}
 & e^{-Lt}(\varepsilon^{-1}t)^{\frac{1}{2}}\left(1+(\varepsilon^{-1}t)^{\frac{\alpha-1}{2}}\right)
 \|\nabla\tilde{D}_\varepsilon[v](t)\|_{L^\infty}\\
 & \le C_2\|v\|_{X_{T,L}}e^{-Lt}(t^{\frac{1}{2}}+t^{\frac{\alpha}{2}})
\left\{\int_0^t e^{Ls}I(s,t-s)\,ds+\int_0^t e^{Ls}(t-s)^{-\frac{1}{2}}I(s,t-s)\,ds\right\}
\end{split}
\end{equation*}
for $t>0$, $L>0$ and $0<\varepsilon<1$.
Similarly to \eqref{eq:3.17}, by Lemma~\ref{Lemma:3.3}, we obtain
\begin{equation}
\label{eq:3.18}
\sup_{0<t<T}
e^{-Lt}(\varepsilon^{-1}t)^{\frac{1}{2}}\left(1+(\varepsilon^{-1}t)^{\frac{\alpha-1}{2}}\right)
\|\nabla\tilde{D}_\varepsilon[v](t)\|_{L^\infty}
\le\frac{1}{4}\|v\|_{X_{T,L}}
\end{equation}
for $0<\varepsilon<1$ and sufficiently large $L\ge1$. 
Combining \eqref{eq:3.17} and \eqref{eq:3.18}, we deduce that in this case 
$$
\|\tilde{D}_\varepsilon[v]\|_{X_{T,L}}\le\frac{1}{2}\|v\|_{X_{T,L}}
$$
for $0<\varepsilon<1$. Thus \eqref{eq:3.8} holds. 
On the other hand, for $v\in X_{T,L}$,
it follows from \eqref{eq:1.14} that 
$F_2\in C(\overline\Omega\times(0,T))$.
Then, applying the parabolic regularity theorem (see e.g. \cite{LSU}, cf. proof of Lemma~\ref{Lemma:2.4}),
we deduce that
$$
\nabla^j\tilde{D}_\varepsilon[v]\in C^\infty(\Omega\times(\tau,T))\cap
BC(\overline\Omega\times(\tau,T))
$$
for any $0<\tau<T$ and $j\in\{0,1\}$.
Therefore we complete the proof of Lemma~\ref{Lemma:3.2}.
\end{proof}

Now we are ready to prove Theorem~\ref{Theorem:1.5}.
\begin{proof}[\textbf{
 Proof of Theorem~\ref{Theorem:1.5}}]
It follows from \eqref{eq:1.10}, \eqref{eq:1.15}, \eqref{eq:2.4} and \eqref{eq:1.17} that 
$$
|\Phi(x)|\le |\varphi(x)|+|[S_2(0)\varphi_b(x)]|\le |x|^{-(N-2)}(M+|\varphi_b|_{L^\infty})
\le |x|^{-\alpha}(M+|\varphi_b|_{L^\infty})
$$
for all $x\in\overline\Omega$. 
Then, by $(G_2)$ we find $c_*>0$ such that 
$$
\|\nabla^j S_1(t)\Phi\|_{L^\infty}\le c_*(M+|\varphi_b|_{L^\infty})(1+t)^{-\frac{\alpha}{2}}h(t)^j
$$
for $t>0$ and $j\in\{0,1\}$ and with $h$ from \eqref{eq:2.1}. 
Let $0<T<\infty$ and $L\ge 1$. Then 
\begin{equation}
\label{eq:3.19}
\begin{split}
 & E_\varepsilon[S_1(\varepsilon^{-1}t)\Phi](t)\\
 & \le c_*(M+|\varphi_b|_{L^\infty})
 \left(1+(\varepsilon^{-1}t)^{\frac{\alpha}{2}}\right)(1+\varepsilon^{-1}t)^{-\frac{\alpha}{2}}\\
 & \qquad
 +c_*(M+|\varphi_b|_{L^\infty})(\varepsilon^{-1}t)^{\frac{1}{2}}\left(1+(\varepsilon^{-1}t)^{\frac{\alpha-1}{2}}\right)
(1+\varepsilon^{-1}t)^{-\frac{\alpha}{2}}h(\varepsilon^{-1}t)\\
 & \le 4c_*(M+|\varphi_b|_{L^\infty})
\end{split}
\end{equation}
for $t>0$ and $0<\varepsilon<1$. 
Furthermore, by Lemma~\ref{Lemma:3.2}, 
taking a sufficiently large $L\ge 1$ if necessary, 
we see that 
\begin{equation}
\label{eq:3.20}
\|\tilde{D}_\varepsilon[v]\|_{X_{T,L}}\le\frac12\|v\|_{X_{T,L}},\qquad v\in X_{T,L},
\end{equation}
for $0<t<T$ and $0<\varepsilon<1$.
For this choice of $L$, on the other hand, by Lemma~\ref{Lemma:2.4} we find $C_L>0$ such that 
\begin{equation}
\label{eq:3.21}
\begin{split}
 & e^{-Lt}E_\varepsilon[D_\varepsilon[\varphi_b]](t)\\
 & \le C|\varphi_b|_{C^{1,\theta}}e^{-Lt}
 \left(1+(\varepsilon^{-1}t)^{\frac{\alpha}{2}}\right)\varepsilon^{\frac{\alpha}{2}}t^{\frac{2-\alpha}{2}}\\
 & \qquad
 +C|\varphi_b|_{C^{1,\theta}}e^{-Lt}
 (\varepsilon^{-1}t)^{\frac{1}{2}}\left(1+(\varepsilon^{-1}t)^{\frac{\alpha-1}{2}}\right)
\varepsilon^{\frac{\alpha}{2}}t^{\frac{2-\alpha}{2}}h(\varepsilon^{-1}t)
 \le C_L|\varphi_b|_{C^{1,\theta}}
\end{split}
\end{equation}
for $t>0$ and $0<\varepsilon<1$. 
Set 
\begin{equation}
\label{eq:3.22}
m:=2\bigg\{4c_*(M+|\varphi_b|_{L^\infty})+C_L|\varphi_b|_{C^{1,\theta}}\bigg\}.
\end{equation}
We deduce from \eqref{eq:3.1}, \eqref{eq:3.19}, \eqref{eq:3.21}, \eqref{eq:3.20} and \eqref{eq:3.22} that 
\begin{equation}
\label{eq:3.23}
\begin{split}
 & \|Q_\varepsilon[v]\|_{X_{T,L}}\\
 & \le\sup_{0<t<T}e^{-Lt}E_\varepsilon[S_1(\varepsilon^{-1}t)\Phi](t)
+\sup_{0<t<T}e^{-Lt}E_\varepsilon[D_\varepsilon[\varphi_b]](t)
+\|\tilde{D}_\varepsilon[v]\|_{X_{T,L}}\\
 & 
\le 4c_*(M+|\varphi_b|_{L^\infty})+C_L|\varphi_b|_{C^{1,\theta}}
+\frac{1}{2}\|v\|_{X_{T,L}}
\le m
\end{split}
\end{equation}
for $v\in X_{T,L}$ with $\|v\|_{X_{T,L}}\le m$ and $0<\varepsilon<1$.
Similarly, we deduce from \eqref{eq:3.20} that
\begin{equation}
\label{eq:3.24}
\left\|Q_\varepsilon[v_1]-Q_\varepsilon[v_2]\right\|_{X_{T,L}}
=\|\tilde{D}_\varepsilon[v_1-v_2]\|_{X_{T,L}}
\le\frac{1}{2}\|v_1-v_2\|_{X_{T,L}}
\end{equation}
for $v_1, v_2\in X_{T,L}$. 
By \eqref{eq:3.23} and \eqref{eq:3.24} 
applying the contraction mapping theorem, for every $\varepsilon \in(0,1)$
we find a unique $v_\varepsilon\in X_{T,L}$ with $\|v_\varepsilon\|_{X_{T,L}}\le m$ 
such that 
$$
v_\varepsilon=Q_\varepsilon[v_\varepsilon]
=S_1(\varepsilon^{-1}t)\Phi
-D_\varepsilon[\varphi_b](t)
+\tilde{D}_\varepsilon[v_\varepsilon](t)
\quad\mbox{in}\quad X_{T,L}.
$$
In particular, it follows from \eqref{eq:3.23} and \eqref{eq:3.22} that with some $C>0$
$$
\|v_\varepsilon\|_{X_{T,L}}\le C(|\varphi_b|_{C^{1,\theta}}+M)\qquad\text{for every } \varepsilon \in(0,1).
$$
Furthermore, 
by $(G_3)$ and Lemmata~\ref{Lemma:2.4},~\ref{Lemma:3.2}
we see that 
$$
\nabla^jv_\varepsilon\in C^\infty(\Omega\times(T_1,T_*))\cap BC(\overline\Omega\times(T_1,T_*))
$$
for any $0<T_1<T_*$ and $j\in\{0,1\}$.

On the other hand, set
$$
w_\varepsilon(x,t)=[S_2(t)\varphi_b](x)+\int_0^t[S_2(t-s)\partial_\nu v_\varepsilon(s)](x)\, ds
$$
for $x\in\overline\Omega$, $t\in(0,T)$ and $\varepsilon \in(0,1)$. 
By \eqref{eq:2.4} and \eqref{eq:3.22} we have that with some $C>0$ 
\begin{equation}
\label{eq:3.25}
\begin{split}
\|w_\varepsilon(t)\|_{L^\infty}
 & \le \|S_2(t)\varphi_b\|_{L^\infty}+\int_0^t\|S_2(t-s)\partial_\nu v_\varepsilon(s)\|_{L^\infty}\,ds\\
 & \le |\varphi_b|_{L^\infty}+\int_0^t|\nabla v_\varepsilon(s)|_{L^\infty}\,ds\\
 & \le \frac{m}{8c_*}+\int_0^t(\varepsilon^{-1}s)^{-\frac{1}{2}}
 \left(1+(\varepsilon^{-1}s)^{\frac{\alpha-1}{2}}\right)^{-1}e^{Ls}\|v_\varepsilon\|_{X_{T,L}}\, ds\\ 
 & \le \frac{m}{8c_*}+\varepsilon^{\frac{\alpha}{2}}e^{LT}m\int_0^ts^{-\frac{\alpha}{2}}\, ds
  \le C(|\varphi_b|_{C^{1,\theta}}+M)<\infty
\end{split}
\end{equation}
for $t\in(0,T)$ and $0<\varepsilon<1$.
Furthermore, 
it follows from \eqref{eq:1.9}, \eqref{eq:1.13} and \eqref{eq:1.14} that
$$
\partial_t w(x,t)=F_1[\varphi_b](x,t)+F_2[v](x,t).
$$
Then,
applying similar arguments as in Lemmata~\ref{Lemma:2.3} and~\ref{Lemma:3.1},
we see that 
$$
\partial_t w\in BC(\overline\Omega\times(T_1,T))
$$
for $0<T_1<T$.
This together with \eqref{eq:2.3} and \eqref{eq:3.25} implies that
$$
\partial_t^\ell\nabla^jw_\varepsilon
\in C^\infty(\Omega\times(T_1,T))\cap BC(\overline\Omega\times(T_1,T))
$$ 
for $0<T_1<T$ and $0\le \ell+j\le 1$.
Therefore we deduce that 
$(v_\varepsilon,w_\varepsilon)$ is a solution of \eqref{eq:1.12} in $\Omega\times(0,T)$. 

Let $(\tilde{v}_\varepsilon,\tilde{w}_\varepsilon)$ be a global-in-time solution of \eqref{eq:1.12} 
satisfying \eqref{eq:1.16}.
Since 
$$
v_\varepsilon-\tilde{v}_\varepsilon
=Q_\varepsilon[v_\varepsilon]-Q_\varepsilon[\tilde{v}_\varepsilon]
=\tilde{D}_\varepsilon[v_\varepsilon-\tilde{v}_\varepsilon]\quad\mbox{in}\quad X_{T,L},
$$
by \eqref{eq:3.8} we have 
$$
\|v_\varepsilon-\tilde{v}_\varepsilon\|_{X_{T,L}}\le\frac{1}{2}\|v_\varepsilon-\tilde{v}_\varepsilon\|_{X_{T,L}}.
$$
This implies that $v_\varepsilon=\tilde{v}_\varepsilon$ in $X_{T,L}$. 
Therefore we deduce that $(v_\varepsilon,w_\varepsilon)$ is a unique solution of \eqref{eq:1.12} 
satisfying \eqref{eq:1.16}.

It remains to prove assertions~(i) and (ii).
Assertion~(i) and \eqref{eq:1.18} immediately follow from $\|v_\varepsilon\|_{X_{T,L}}\le m$, \eqref{eq:3.22} and \eqref{eq:3.25}.  
On the other hand,
by \eqref{eq:3.25} we have
\begin{equation*}
\begin{split}
&
\|w_\varepsilon(t)-S_2(t)\varphi_b\|_{L^\infty}
\le
\int_0^t\|S_2(t-s)\partial_{x_N}v_\varepsilon(s)\|_{L^\infty}\, ds
\\
&\quad
\le
\int_0^t|\nabla v_\varepsilon(s)|_{L^\infty}\,ds
\le C\|v_\varepsilon\|_{X_{T,L}}\varepsilon^{\frac{\alpha}{2}}(1+T^{\frac{2-\alpha}{2}}e^{LT})
\end{split}
\end{equation*}
for all $t\in(0,T)$.
This implies \eqref{eq:1.19}.
Thus assertion~(ii) follows,
and the proof of Theorem~\ref{Theorem:1.5} is complete.
\end{proof}

\noindent
\begin{proof}[\textbf{Proof of Corollary~\ref{Corollary:1.1}.}]
Corollary~\ref{Corollary:1.1} immediately follows from Theorem~\ref{Theorem:1.5} and Definition~\ref{Definition:1.1}.
\end{proof}

%
\section{Estimates from below}
%
\begin{lemma}
\label{Lemma:4.1}
Let $\varepsilon\in(0,1)$, $\varphi_b\equiv0$ and $\varphi$ be nonnegative and satisfy \eqref{eq:1.15}.
Let $z$ be a solution of
\begin{equation}
\label{eq:4.1}
\partial_t z-\Delta z=0\quad\mbox{in $\Omega\times(0,\infty)$},
\qquad
z=0\quad \mbox{on $\partial\Omega\times(0,\infty)$},
\qquad
z(\cdot,0)=\varphi\quad \mbox{in $\Omega$}.
\end{equation}
Set $\underline{u}_\varepsilon(x,t)=z(x, \varepsilon^{-1}t)$ for $(x,t)\in\Omega\times[0,\infty)$.
Then the solution $u_\varepsilon$ of \eqref{eq:1.1} satisfies
\begin{equation}
\label{eq:4.2}
u_\varepsilon(x,t)\ge\underline{u}_\varepsilon(x,t),\qquad (x,t)\in\Omega\times(0,\infty).
\end{equation}
\end{lemma}
\begin{proof}
By nonnegativity of $z$ in $\Omega$ and the homogeneous Dirichlet boundary condition in \eqref{eq:4.1},
we see that $\partial_\nu z\le 0$ on $\partial\Omega$.
>From \eqref{eq:4.1} we conclude that $\underline u_\varepsilon$ solves
$$
\varepsilon\partial_t\underline u_\varepsilon-\Delta\underline u_\varepsilon =0\quad\mbox{in $\Omega\times(0,\infty)$},
\qquad
\underline u_\varepsilon(\cdot,0)=\varphi\quad\mbox{in $\Omega$}
$$
and
$$
\partial_t\underline u_\varepsilon(x,t)+\partial_\nu\underline u_\varepsilon(x,t)
=\varepsilon^{-1}\partial_tz(x,\varepsilon^{-1}t)+\partial_\nu z(x,\varepsilon^{-1}t)
=\partial_\nu z(x,\varepsilon^{-1}t)\le 0
$$
on $\partial\Omega\times(0,\infty)$.
Therefore, $\underline u_\varepsilon$ is a subsolution of \eqref{eq:1.1},
while $u_\varepsilon$ is a supersolution,
and applying the comparison principle (see \cite[Theorem~2.2]{vBD}), we obtain \eqref{eq:4.2}.
\end{proof}
\begin{lemma}
\label{Lemma:4.2}
Let $b>1$ and put
\begin{equation}
\label{eq:4.3}
\varphi(x)=|x|^{2-N}\chi_{\{|x|>b\}},\qquad x\in\Omega.
\end{equation}
Then for any compact set $\mathfrak K_*\subset\Omega$ and $\tau>0$
there exists $C>0$ such that the solution $z$ of \eqref{eq:4.1} satisfies
\begin{equation}
\label{eq:4.4}
z(x,t)\ge Ct^{-\frac{N}{2}+1},\qquad x\in\mathfrak K_*,\quad t>\tau.
\end{equation}
\end{lemma}
\begin{proof} 
Let $\mathfrak K_*$ be a compact set in $\Omega$.
Then we can take $a\in(1,b)$ and $\beta>0$ such that $a<|x|\le\beta b$ for all $x\in\mathfrak K_*$.
Since it follows form \cite[Theorem~1.1]{GS} that there are $C_1>0$ and $C_2>0$ such that
$$
\Gamma_D(x,y,t)\ge C_1t^{\frac{N}{2}}\exp\bigg(-C_2\frac{|x-y|^2}{t}\bigg)
$$
for all $x,y\in\Omega$ with $|x|>a$, $|y|>a$ and $t>0$,
by \eqref{eq:4.3} we have $C_3>0$ satisfying
\begin{align*}
z(x,t)=\int_\Omega\Gamma_D(x,y,t)\varphi(y)\,dy
&
=\int_b^\infty\int_{\mathbb S^{N-1}}r^{2-N}\Gamma_D(x,r\omega,t)r^{N-1}\,d\omega\,dr
\\
&
\ge
C_1\int_b^\infty\int_{\mathbb S^{N-1}}rt^{-\frac{N}{2}}\exp\bigg(-C_2\frac{|x-r\omega|^2}{t}\bigg)\,d\omega\,dr
\\
&
\ge
C_1|\mathbb{S}^{N-1}|t^{-\frac{N}{2}}\int_b^\infty r\exp\bigg(-C_2\frac{(1+\beta)^2r^2}{t}\bigg)\,dr
\\
&
=
C_3t^{-\frac{N}{2}+1}(1+\beta)^{-2}\exp\bigg(-C_2\frac{(1+\beta)^2b^2}{t}\bigg)
\end{align*}
for all $x\in\Omega$ with $a<|x|\le\beta b$ and $t>0$.
This implies \eqref{eq:4.4}, thus Lemma~\ref{Lemma:4.2} follows.
\end{proof}

Now we are ready to prove Theorem~\ref{Theorem:1.6}.
\begin{proof}[\textbf{Proof of Theorem~\ref{Theorem:1.6}}]
Let $\mathfrak K$ be a compact set in $\Omega$ such that
$\mathfrak K\subset\mathfrak K_*\times[t_1,t_2]$ for some compact set $\mathfrak K_*\subset\Omega$ 
and $0<t_1<t_2<\infty$,
and let $\varphi$ be as in Lemma~\ref{Lemma:4.2}.
Then, applying Lemma~\ref{Lemma:4.1} and Lemma~\ref{Lemma:4.2} to $\tau:=t_1$, 
we see that there exists $C_*>0$ such that
$$
u_\varepsilon(x,t)\ge z(x,\varepsilon^{-1}t)\ge C_*(\varepsilon^{-1}t)^{-\frac{N}{2}+1}
\ge C_*t_2^{-\frac{N}{2}+1}\varepsilon^{\frac{N}{2}-1}
$$
for all $(x,t)\in\mathfrak K$.
This implies \eqref{eq:1.20}, and the proof of Theorem~\ref{Theorem:1.6} is complete.
\end{proof}

\noindent
{\bf Acknowledgment.}
The first author was supported in part by the Slovak Research and Development Agency 
under the contract No. APVV-18-0308 and by VEGA grant 1/0347/18.
The second and third authors of this paper were supported in part 
by JSPS KAKENHI Grant Numbers JP 19H05599.
The third author was also supported in part by JSPS KAKENHI Grant Numbers JP 16K17629 
and JP 20K03689.


\end{document}